\documentclass[]{amsart}
\usepackage{nima}
\usepackage{xnotes}
\usepackage{pdfa}
\usepackage{amsaddr}

\usepackage{tikz-cd}
\usepackage{etoolbox}
\apptocmd{\sloppy}{\hbadness 10000\relax}{}{}

\newcommand{\tablabel}[1]{\label{tab:#1}}
\newcommand{\tabref}[1]{\ref{tab:#1}}
\newcommand{\Tabref}[1]{Table~\tabref{#1}}

\renewcommand{\sth}{:}

\title{Bisimplicial Complexes and Asphericity}
\author{Nima Hoda}
\address{Instytut Matematyczny,
	Uniwersytet Wroc\l awski\\
	pl.\ Grun\-wal\-dzki 2/4,
	50--384 Wroc{\l}aw, Poland}
\email{nima.hoda@mail.mcgill.ca}
\date{\today}

\newcommand{\nabladelta}{\protect\substack{\nabla \vspace{-1.1pt} \\
    \hspace{0.4pt}\Delta}}
\newcommand{\bs}{\raisebox{1.5pt}{%
    \protect\scalebox{1}[0.75]{\ensuremath{\nabladelta}}}}

\DeclareMathOperator{\lk}{lk}
\DeclareMathOperator{\CAT}{CAT}
\newcommand{\join}{\bowtie}

\begin{document}

\begin{abstract}
  We present a discrete Morse-theoretic method for proving that a
  regular CW complex is homeomorphic to a sphere.  We use this method
  to define \texorpdfstring{\emph{bisimplices}}{bisimplices}, the
  cells of a class of regular CW complexes we call
  \texorpdfstring{\emph{bisimplicial complexes}}{bisimplicial
    complexes}.  The \texorpdfstring{$1$-skeleta}{1-skeleta} of
  bisimplices are complete bipartite graphs making them suitable in
  constructing higher dimensional skeleta for bipartite graphs.  We
  show that the flag bisimplicial completion of a finite bipartite
  bi-dismantlable graph is collapsible.  We use this to show that the
  flag bisimplicial completion of a quadric complex is contractible
  and to construct a compact \texorpdfstring{$K(G,1)$}{K(G,1)} for a
  torsion-free quadric group \texorpdfstring{$G$}{G} .
\end{abstract}

\maketitle

\tableofcontents

\section{Introduction}

CW complexes are typically constructed by gluing together Euclidean
polyhedra along faces.  A \defterm{Euclidean polyhedron} is the convex
hull of a finite point set in a Euclidean space, e.g., simplices and
cubes.  However, not all CW structures on cells of a CW complex arise
as Euclidean polyhedra \cite{Kalai:1988} and for some applications it
is natural to use nonpolyhedral cells.  In this paper we construct an
infinite family of nonpolyhedral CW balls called \emph{bisimplices}.
The $1$-skeleta of bisimplices are connected complete bipartite graphs
and so we consider them as bipartite analogs of simplices.  Our
motivation for this construction is to find a natural contractible
higher dimensional skeleton for quadric complexes.

\newterm{Quadric complexes} are locally finite simply connected square
complexes satisfying a certain combinatorial nonpositive curvature
condition.  A group is \newterm{quadric} if it acts properly and
cocompactly on a quadric complex.  Quadric complexes are examples of
the generalized $(4,4)$-complexes of Wise \cite{Wise:2003} and were
first studied in detail in the context of geometric group theory by
the present author \cite{Hoda:2017}.  They generalize the folder
complexes of Chepoi \cite{Chepoi:2000} and may be considered as square
analogs of the $2$-skeleta of systolic complexes
\cite{Januszkiewicz:2006}.  They can be characterized by their
$1$-skeleta, which are precisely the hereditary modular graphs of
metric graph theory \cite{Bandelt:1988}.

\subsection{Summary of Results}

\begin{figure}
  \centering
  \begin{tikzcd}
    \vcenter{\hbox{\includegraphics[scale=0.5]{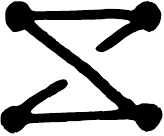}}}
    \arrow[d,dashed,"\substack{\, \\ \text{span} \\ \text{\&} \\ \text{cone} \\ \,}"']
    &\vcenter{\hbox{\includegraphics[scale=0.5]{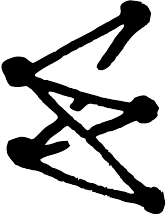}}}
    \arrow[d,dashed,"\substack{\, \\ \text{span} \\ \text{\&} \\ \text{cone} \\ \,}"']
    & \cdots
    &\vcenter{\hbox{\includegraphics[scale=0.4]{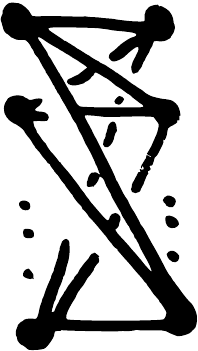}}}
    \arrow[d,dashed,"\substack{\, \\ \text{span} \\ \text{\&} \\ \text{cone} \\ \,}"']
    \\
    \vcenter{\hbox{\includegraphics[scale=0.5]{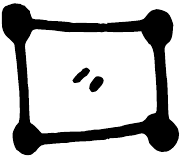}}}
    &\vcenter{\hbox{\includegraphics[scale=0.35]{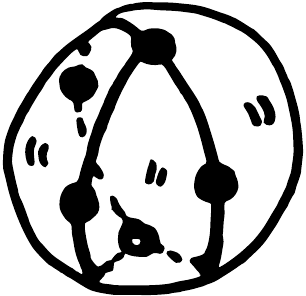}}}
    & \cdots
    & \vcenter{\hbox{\includegraphics[scale=0.3]{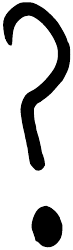}\includegraphics[scale=0.3]{figsdrawn/qmark.pdf}\includegraphics[scale=0.3]{figsdrawn/qmark.pdf}}}
  \end{tikzcd}
  \caption{Bisimplices are essentially constructed by starting with a
    $K_{m,n}$, $m,n \ge 2$, inductively spanning a biclique on each
    proper $K_{m',n'}$ subgraph, $m',n' \ge 2$, and then taking the
    cone of the result.  The difficulty lies in showing that the base
    of this cone is homeomorphic to $\Sph^{m+n-3}$ and hence that the
    cone has the structure of a regular CW complex with a single top
    dimensional cell of dimension $m+n-2$.  This is trivial for
    $(m,n)$ equal to $(2,2)$ or $(2,3)$, as seen in the figure.  To
    prove it for general $(m,n)$ is not quite so easy.}
  \figlabel{lowdbs}
\end{figure}

In contrast to simplices which are indexed by dimension, bisimplices
are indexed by two natural numbers $m,n \ge 1$.  For each dimension
$d \ge 2$ there are $\bigl\lceil\frac{d-1}{2}\bigr\rceil$ bisimplices
of dimension $d$.  Recall that a CW complex is \newterm{regular} if
the characteristic maps of its cells are injective.  See
\Figref{lowdbs}.

\begin{mainthm}[\Thmref{bsbd}]
  \mainthmlabel{ibisimplices} There exists a family
  $\{ \bs^{m,n} \}_{m,n \ge 1}$ of regular CW complexes called
  bisimplices satisfying the following conditions.
  \begin{itemize}
  \item $\bs^{m,n}$ has a unique maximal cell and so $\bs^{m,n}$ is
    homeomorphic to a ball.
  \item $\bs^{m,n}$ has dimension $m+n$.
  \item The $1$-skeleton of $\bs^{m,n}$ is the complete bipartite
    graph $K_{m+1,n+1}$.
  \end{itemize}
\end{mainthm}

Moreover, the cells of a bisimplex $\bs$ are also bisimplices and
these cells are precisely the full subcomplexes of $\bs$, aside from a
few degenerate cases such as the $K_{0,\ell}$ and $K_{1,\ell}$
subgraphs of the $1$-skeleton.  We consider vertices and edges to be
bisimplices also.  These properties uniquely determine the cell posets
of the bisimplices.  However, proving that this family of posets is
indeed a family of cell posets is not at all trivial and is an
interesting application of the discrete Morse theory of Forman
\cite{Forman:1998} and the Generalized Poincar{\'e} Conjecture.
Specifically, we prove \Mainthmref{ibisimplices} by inductively applying
the following theorem, which we expect to have applications elsewhere.

\begin{mainthm}[\Thmref{plsphere} and \Rmkref{lpsphere}]
  \mainthmlabel{ilpsphere} Let $P$ be a poset such that the order
  complexes of the under sets of $P$ are PL-triangulated spheres.  If
  $P$ and all of its over sets admit spherical matchings then the
  order complex of $P$ is a PL-triangulated sphere.
\end{mainthm}

A \newterm{spherical matching} is a combinatorial structure on the
Hasse diagram of the cell poset of a regular CW complex $X$.  This
combinatorial structure is essentially a discrete Morse function on
$X$ having exactly two critical cells and so, by the Sphere Theorem of
Forman \cite{Forman:1998}, implies that $X$ is homotopy equivalent to
a sphere.  See \Secref{forman} for an introduction to discrete Morse
theory and the definition of spherical matching.  The other
terminology of \Mainthmref{ilpsphere} is defined in \Secref{cellposet}.

Having constructed the family of bisimplices, we may construct regular
CW complexes having bisimplices as cells.  We call these
\newterm{bisimplicial complexes} when the intersection of any two
bisimplices is a full subcomplex.  Given a bipartite graph $\Gamma$
there is a natural bisimplicial complex $\bs(\Gamma)$ called the
\newterm{flag bisimplicial completion} of $\Gamma$.  The flag
bisimplicial completion is defined analogously to the flag simplicial
completion, also known as the clique complex, of a graph.

Our primary motivation for the definition of the flag bisimplicial
completion is to apply it to the bipartite $1$-skeleta of quadric
complexes.  \defterm{Quadric complexes} may be defined as simply
connected $2$-dimensional CW complexes whose minimal area disk
diagrams are $\CAT(0)$ square complexes.  We would like a natural way
to glue higher dimensional cells to a quadric complex to obtain a
contractible supercomplex.  The $1$-skeleton of a quadric complex $X$
is bipartite and may contain $K_{2,3}$ so it is not possible to extend
$X$ simplicially or cubically.  However, $X$ equals the $2$-skeleton
of $\bs(X^1)$, so a natural candidate for a contractible supercomplex
is $\bs(X^1)$.

\begin{mainthm}[\Thmref{quadcont}]
  \mainthmlabel{ifcquadcont} Let $X$ be a nonempty quadric complex.  Then
  the flag bisimplicial completion $\bs(X^1)$ is contractible.
\end{mainthm}

Metric balls in $X^1$ induce finite quadric subcomplexes of $X$ and
finite quadric complexes have bi-dismantlable $1$-skeleta
\cite{Bandelt:1988, Hoda:2017}.  A finite bipartite graph is
\newterm{bi-dismantlable} if it can be reduced to a nonempty connected
complete bipartite graph by successively deleting a vertex whose
neigbourhood is contained in the neighborhood of another vertex.
\Mainthmref{ifcquadcont} then follows from \Mainthmref{idismcont}
below whose proof is another application of the discrete Morse theory
of Forman.

\begin{mainthm}[\Thmref{dismclps}]
  \mainthmlabel{idismcont} Let $X$ be a flag, nonempty finite bisimplicial
  complex.  If $X^1$ is bipartite and bi-dismantlable then $X$ is
  collapsible.
\end{mainthm}

This method of proving contractibility mirrors that of Chepoi and
Osajda for weakly systolic complexes \cite{Chepoi:2015} via
\newterm{LC-contractibility} \cite{Civan:2007, Matousek:2008}.

As pointed out to the present author by Damian Osajda, a quadric
complex $X$ may also naturally be made contractible by extending each
connected complete bipartite subgraph of $X^1$ to a complete subgraph
and then taking the flag simplicial completion of the resulting graph.
However, this operation preseves neither the $1$-skeleton nor the
$2$-skeleton of $X$.  Moreover, the resulting complex has higher
dimension than the flag bisimplicial completion $\bs(X^1)$.

If $X$ is a compact locally quadric complex, the construction of the
bisimplicial completion of the universal cover $\univcov X$ has a
corresponding construction in the base.  We obtain from $X$ a compact
complex $X^{+}$ whose $2$-skeleton is $X$ and whose higher cells are
obtained by successively gluing in higher dimensional bisimplices
along immersions of their boundaries.  Then applying
\Mainthmref{ifcquadcont} we obtain the following.

\begin{mainthm}[\Thmref{quadkgone}]
  \mainthmlabel{ikgone} Let $X$ be a compact locally quadric complex.  If
  $\pi_1(X)$ is torsion-free then $X^{+}$ is a compact
  $K(\pi_1(X),1)$.
\end{mainthm}

Note that $\pi_1(X)$ in \Mainthmref{ikgone} is torsion-free if and only if
the automorphism group of every immersion of the $2$-skeleton of a
bisimplex into $X$ is trivial.  This is a consequence of the invariant
biclique theorem for quadric complexes \cite{Hoda:2017}.  Moreover,
every torsion-free quadric group is the fundamental group of some
locally quadric complex.

\subsection{Structure of the Paper}

In \Secref{cellposet} we give some background on posets, regular CW
complexes and PL-triangulated spheres.  In \Secref{forman} we present
basic theorems of the discrete Morse theory of Forman.  We apply these
theorems and the topological Generalized Poincar{\'e} Conjecture to
prove a discrete Morse-theoretic sphere recognition theorem.  We use
our sphere recognition theorem in \Secref{bisimplex} to construct the
infinite family of bisimplices.  We then prove some basic facts about
bisimplices.  In \Secref{bscomplex} we introduce bisimplicial
complexes and prove that flag finite bisimplicial complexes with
dismantlable $1$-skeleta are collapsible, again making use of discrete
Morse theory.  Finally, in \Secref{quadric} we prove that the flag
bisimplicial completion of a quadric complex is contractible and
describe how to construct a $K(G,1)$ for a torsion-free quadric group
$G$.

\subsection*{Acknowledgements}
The author would like to acknowledge Daniel T. Wise for many
invaluable discussions throughout the course of this research and for
his encouragement from its outset.  Thanks also to Piotr Przytycki for
a very detailed list of comments and corrections for a version of this
paper included in the author's thesis.  This research was supported in
part by an NSERC Postgraduate Scholarship, a Pelletier Fellowship, a
Graduate Mobility Award, an ISM Scholarship and the NSERC grant of
Daniel T. Wise.

\section{Posets and Regular CW Complexes}
\seclabel{cellposet}

Let $P$ be a poset.  The \defterm{covering relation} $C_P$ on $P$ is
the following binary relation.
\[ C_P(x,y) \quad \Longleftrightarrow \quad \text{$x < y$ and there is
    no $z$ satisfying $x<z<y$} \] A poset $P$ is \defterm{graded} if
every element $x \in P$ is assigned a grade $|x| \in \N$ such that the
following conditions hold.
\begin{align*}
  C_P(x,y) \quad &\Longrightarrow \quad |x| + 1 = |y| \\
  x < y \quad &\Longrightarrow \quad |x| < |y|
\end{align*}

Let $P$ be a poset.  For $x \in P$, the \defterm{over set} $O_x$ and
\defterm{under set} $U_x$ \defterm{of} $P$ \defterm{at} $x$ are the
following subsets of $P$.
\[ O_x = \{y \in P \sth y > x \} \qquad U_x = \{y \in P \sth y < x
\} \]
We may write $O^P_x$ and $U^P_x$ if the poset is not clear from the
context.  Note that making the inequalities in the definitions of
$O_x$ and $U_x$ nonstrict would give what are usually referred to as
the \emph{principal ideal} and \emph{principal filter} having
\emph{principal element} $x$.  For $x,y \in P$, the \defterm{strict
  interval} $(x,y)$ of $P$ \defterm{between} $x$ and $y$ is the subset
of $P$ defined as follows.
\[ (x,y) = O_x \cap U_y \]
The over sets, under sets and strict intervals of $P$ are themselves
posets by restricting the order relation.  If $P$ is graded then the
over sets, under sets and strict intervals of $P$ are likewise
themselves graded.

Let $P$ be a poset.  The set of nonempty chains of $P$ form an
abstract simplicial complex.  Its associated simplicial complex is the
\defterm{order complex} $\Delta_P$ of $P$.

A CW complex is \defterm{regular} if the characteristic maps of its
cells are embeddings.  Let $X$ be a regular CW complex.  The cells of
$X$ are regular CW subcomplexes.  Viewing a cell $x$ as a ball, we
denote its boundary by $\bd x$ and its interior by $x^{\circ}$.  The
\defterm{$k$-skeleton} $X^k$ of $X$ is the subcomplex of $X$ formed by
the union of the cells of $X$ of dimension at most $k$.  The
\defterm{cell poset} $P_X$ of $X$ is the set of cells of $X$ ordered
by inclusion.  Cell posets are equipped with a natural grading, namely
dimension: $|x| = \dim x$.  A cell poset $P$ uniquely determines its
regular CW complex $X_P$.  The order complex of the cell poset of a
regular CW complex $X$ is isomorphic to the barycentric subdivision of
$X$.  A subset $Q$ of $P$ is the cell poset of a subcomplex of $X_P$
iff $Q$ is \defterm{downward closed}, meaning the following.
\[ \text{$x \in Q$ and $y < x$} \quad \Longrightarrow \quad y \in Q \]

The following theorem of Bj\"{o}rner characterizes the cell posets.
\begin{thm}[{Bj\"{o}rner \cite[Proposition~3.1]{Bjorner:1984}}]
  \thmlabel{cellpos} Let $P$ be a poset.  Then $P$ is a cell poset iff
  the order complexes of its under sets are homeomorphic to spheres.
\end{thm}
\begin{proof}
  If $P$ is the cell poset of a regular CW complex $X$ then its under
  sets are the cell posets of the boundaries of its cells.  The order
  complexes of these under sets are the barycentric subdivisions of
  these cells and so are homeomorphic to spheres.

  To prove the converse, we construct $X_P$ inductively on dimension.
  Define the height function $h \colon P \to \N$ as follows.
  \[ h(x) = \max\bigl\{|C|-1 \sth \text{$C$ is a chain in $P$ with
      maximum $x$}\bigr\} \] Note that $h(x)$ is finite because,
  otherwise, the order complex of the under set at $x$ would be
  infinite dimensional.  We have that $h(x) = 0$ for minimal elements
  of $P$ and that $h(x) < h(y)$ for $x < y$.  We will show that $h$ is
  a grading on $P$.  Let $x \in P$.  Since the under set $U_x$ has
  order complex homeomorphic to a sphere $\Sph^{\ell}$, the maximal
  chains of $U_x$ must all have size $\ell+1$.  Hence $h(x) = \ell+1$
  and $h(y) = \ell$ for any element of $P$ that is covered by $x$.

  Let $P^k \subseteq P$ be defined by the following.
  \[ P^k = \{x \in P \sth h(x) \le k \} \] We will construct $X_P$
  such that its $k$-skeleton $X_P^k$ has cell poset $P^k$.  We begin
  the induction by letting $X_P^0 = P^0$.  Suppose we have constructed
  $X_P^k$ having cell poset $P^k$.  Let $x \in P$ with $h(x) = k+1$.
  Since $h$ is a grading, we have $U_x \subseteq P^k$.  So $U_x$ is
  the cell poset of a subcomplex $A$ of $X_P^k$.  The barycentric
  subdivision of $A_x$ is isomorphic to the order complex of $U_x$
  which, by hypothesis, is homeomorphic to a sphere.  This sphere has
  dimension $k$ since that is the height of a maximal chain in $U_x$.
  We construct $X_P^{k+1}$ from $X_P^k$ by attaching a $(k+1)$-ball
  along its boundary to each $A_x$ with $h(x) = k+1$.  Then $P^{k+1}$
  is the cell poset of $X_P^{k+1}$.  Having inductively defined the
  skeleta
  \[ X_P^0 \subseteq X_P^1 \subseteq X_P^2 \subseteq \cdots \] we
  obtain $X_P$ as the colimit.
\end{proof}

We consider the empty space to be the sphere of dimension $-1$.

Let $P$ be a cell poset.  The under sets of $P$ are also cell posets.
More precisely, the under set $U_x$ at a cell $x$ is the cell poset
of the regular CW complex structure on the boundary of $x$.

The order complex of the over set $O_x$ is isomorphic to the link of
the barycenter of $x$ in the barycentric subdivision of $X_P$.
However, because of the existence of homology spheres, $O_x$ need not
be a cell poset.  A homology sphere is a manifold with the homology of
a sphere but which is not homeomorphic to a sphere.  The double
suspension of a homology sphere is homeomorphic to a sphere, as first
proved in full generality by Cannon \cite{Cannon:1979}.  The
Poincar{\'e} homology sphere $X$, also known as the spherical
dodecahedron space, is a homology $3$-sphere that has a simplicial
triangulation \cite[Section~62]{Seifert:1934}.  Let $B$ be the regular
CW complex with a single cell of dimension 6 and whose boundary
$\bd B$ has the structure of the simplicial double suspension of $X$.
Let $e$ be a $1$-simplex of $B$ joining two of the suspension points
of $\bd B$.  Then the link $\lk e$ is isomorphic to $X$ and so the
over set $O_e$ of $e$ in $P_B$ is the cell poset of $X$ augmented with
a new maximum element corresponding to the top-dimensional cell of
$B$.  The under set of $B$ in $O_e$ then has order complex
homeomorphic to $X$ and not a sphere and hence $O_e$ is not a cell
poset.

If $X_P$ is a simplicial complex, then the over set $O_x$ at a cell
$x$ is also a cell poset.  In fact, $O_x$ is the cell poset of the
link $\lk x$ of $x$.  \Thmref{links} characterizes the cell posets in
which this holds.

\begin{prop}
  \proplabel{lkmfold} Let $P$ be a cell poset and suppose $X_P$ is
  connected.  If the over sets of $P$ at its minimal elements have
  order complexes homeomorphic to spheres then $X_P$ is a manifold.
\end{prop}
\begin{proof}
  The order complexes of the over sets of $P$ at minimal elements
  are isomorphic to the links of vertices of the order complex of $P$.
  Hence the links of vertices of the barycentric subdivision of $X_P$
  are homeomorphic to spheres.  That $X_P$ is connected ensures that
  these spheres all have the same dimension, say $d-1$, and that $X_P$
  is second countable.  Then $X_P$ is a $d$-manifold.
\end{proof}

A \defterm{PL-triangulated manifold} is a simplicial complex $X$ that
is homeomorphic to a manifold such that the link of every simplex of
$X$ is homeomorphic to a sphere \cite{Hudson:1969}.  PL-triangulated
manifolds are referred to as \emph{combinatorial manifolds} in the
PL-topology literature.

\begin{prop}
  \proplabel{combmfold} Let $P$ be a cell poset and suppose $X_P$ is
  connected.  Then the order complex of $P$ is a PL-triangulated
  manifold iff the order complexes of the strict intervals and over
  sets of $P$ are homeomorphic to spheres.
\end{prop}

\Propref{combmfold} is an immediate consequence of the following lemma.

\begin{lem}
  \lemlabel{plmfoldlem} Let $P$ be a cell poset and suppose $X_P$ is
  connected.  Let $P'$ be the cell poset of the order complex
  $\Delta_P$.  The order complexes of the over sets of $P'$ are
  homeomorphic to spheres iff the order complexes of the strict
  intervals and over sets of $P$ are homeomorphic to spheres.
\end{lem}

Indeed, the over sets of $P'$ are the cell posets of the links of the
simplices of $\Delta_P$.  If $\Delta_P$ is a PL-triangulated manifold
then these are all homeomorphic to spheres.  Conversely, if the links
are all homeomorphic to spheres then since $X_P$ is connected,
\Propref{lkmfold} ensures that $\Delta_P$ is homeomorphic to a
manifold.  Then, by definition, $X_P$ is a PL-triangulated manifold.

Before proving \Lemref{plmfoldlem} we study the over sets of cell
posets of order complexes.  Let $P$ be a poset and let $P'$ be the
cell poset of the order complex of $P$.  Every $c \in P'$ is a
nonempty chain
\[ c = \{ c_0, c_1, \ldots, c_k \} \subset P \] with
\[ c_0 < c_1 < \cdots < c_k \]
in $P$ and these chains are ordered by inclusion.  Each element of the
over set $O_c$ is a chain in $P$ containing $c$ and so is
determined by its intersections with $U_{c_0}$, with $O_{c_k}$ and
with the strict intervals $(c_{i-1},c_i)$.  It follows that $O_c$
embeds in the componentwise product order
\[ (U_{c_0})'_{\bot} \times (c_0,c_1)'_{\bot} \times (c_1,c_2)'_{\bot}
\times \cdots \times (c_{k-1},c_k)'_{\bot} \times (O_{c_k})'_{\bot} \]
where $Q'_{\bot}$ denotes the poset of \emph{all} chains (including
the empty chain) in a poset $Q$.  Aside from the presence of a minimum
element corresponding to the empty simplex, this product is isomorphic
to the cell poset of the simplicial join of the order complexes of
$U_{c_0}$, $O_{c_k}$ and the $(c_{i-1},c_i)$.  The complement of this
minimum element is the image of $O_c$ under its embedding in the
product.  Hence, $X_{O_c}$ is isomorphic to the simplicial join
\[ X_{O_c} \isomor O_{U_{c_0}} \join O_{(c_0,c_1)} \join
O_{(c_1,c_2)} \cdots \join O_{(c_{k-1},c_k)} \join
O_{O_{c_k}} \]
of the order complexes of $U_{c_0}$, $O_{c_k}$ and the
$(c_{i-1},c_i)$.

\begin{proof}[Proof of \Lemref{plmfoldlem}]
  The joins of spheres are spheres so, by the discussion above, the
  ``if'' part of \Lemref{plmfoldlem} has been established.  It remains
  to prove the ``only if'' part.

  Assume that the order complexes of the over sets of $P'$ are
  homeomorphic to spheres.  Let $x,y,z \in P$ with $x < y$.  Let
  $c^x \in P'$ and $c^z \in P'$ be maximal chains of $P$ that have $x$
  and $z$ as their maximums.  Let $c_y \in P'$ be a maximal chain of
  $P$ that has $y$ as its minimum.  Then we have
  \[ X_{O_{c_y \cup c^x}} \isomor O_{(x,y)} \] and
  \[ X_{O_{c^z}} \isomor \Delta_{O_z} \] and so $O_{(x,y)}$ and
  $O_{O_z}$ are homeomorphic to spheres.
\end{proof}

A \defterm{PL-triangulated sphere} is a PL-triangulated manifold that
is homeomorphic to a sphere.

\begin{thm}
  \thmlabel{links} Let $P$ be a cell poset.  The following conditions are
  equivalent.
  \begin{enumerate}
  \item \itmlabel{plcells} The order complexes of under sets of $P$
    are PL-triangulated spheres.
  \item \itmlabel{csphints} The order complexes of strict intervals of
    $P$ are PL-triangulated spheres.
  \item \itmlabel{sphints} The order complexes of strict intervals of
    $P$ are homeomorphic to spheres.
  \item \itmlabel{gsets} The over sets of $P$ are cell posets.
  \end{enumerate}
\end{thm}

\Thmref{links} characterizes the regular CW complexes $X$ for which
each $d$-cell $x$ may be associated a link having the structure of a
regular CW complex in which the $(k-d-1)$-cells naturally correspond
to the $k$-cells of $X$ that are incident to $x$.  \Thmref{links} says
that this holds precisely when the boundaries of the cells of $X$ are
PL-triangulated spheres.

\begin{proof}[Proof of \Thmref{links}]
  $\text{(\itmref{plcells})} \Longrightarrow
  \text{(\itmref{csphints})}$ Let $(x,y)$ be a strict interval of $P$.
  Then $(x,y)$ is the over set of $U_y$ at $x$.  So, by the forward
  implication of \Propref{combmfold}, the order complex of $(x,y)$ is
  homeomorphic to a sphere.  Every over set and strict interval of
  $(x,y)$ is a strict interval of $P$ and so, by the same argument,
  must have order complex homeomorphic to a sphere.  Then, by the
  reverse implication of \Propref{combmfold}, the order complex of
  $(x,y)$ is a PL-triangulated sphere.
  
  $\text{(\itmref{csphints})} \Longrightarrow \text{(\itmref{sphints})}$
  This is clear.
  
  $\text{(\itmref{sphints})} \Longrightarrow \text{(\itmref{gsets})}$ Let
  $O_x$ be a over set of $P$.  By \Thmref{cellpos} we need only
  show that the under set of $O_x$ at any $y \in O_x$ has order
  complex homeomorphic to a sphere.  But the under set of $O_x$ at
  $y$ is the strict interval $(x,y)$ of $P$ and so this holds.
  
  $\text{(\itmref{gsets})} \Longrightarrow \text{(\itmref{plcells})}$
  Let $U_z$ be a under set of $P$.  By \Thmref{cellpos}, the order
  complex of $U_z$ is homeomorphic to a sphere.  To prove that it is a
  PL-triangulated sphere it suffices, by \Propref{combmfold}, to show
  that, for $x < y < z$, the strict interval $(x,y)$ and the over
  set $O^{U_y}_x$ of $U_y$ at $x$ have order complexes homeomorphic to
  spheres.  But $(x,y)$ and $O^{U_y}_x$ are equal to the under sets
  $U^{O_x}_z$ and $U^{O_x}_y$ of $O_x$ and so, by \Thmref{cellpos},
  they have order complexes homeomorphic to spheres.
\end{proof}

\section{Forman Morse Theory and the Recognition of Spheres}
\seclabel{forman}

Let $P$ be a finite graded poset.  The \defterm{Hasse diagram}
$\Gamma_P$ of $P$ is the covering relation $C_P$ viewed as a directed
graph.  A \defterm{matching} $M$ on $P$ is a set of pairwise disjoint
closed edges of $\Gamma_P$.  An element $x \in P$ is \defterm{matched}
by $M$ if it is contained in an edge of $M$.  A matching $M$ on $P$ is
\defterm{acyclic} if the directed graph $\Gamma_P^M$ obtained from
$\Gamma_P$ by reversing the direction on the edges of $M$ has no
directed cycles.  An element $x \in P$ is a \defterm{critical element}
of $M$ if $x$ is not matched by $M$.  Acyclic matchings are also known
as \defterm{Morse matchings}.  If $P$ is a cell poset then an acyclic
matching $M$ on $P$ determines a Forman discrete Morse function
\cite{Forman:1998} on $X_P$ with the same set of critical cells
\cite{Chari:2000}.  The language of acyclic matchings for discrete
Morse theory is due to Chari \cite{Chari:2000}.

\begin{prop}
  \proplabel{alternate} Let $P$ be a finite graded poset, let $M$ be a
  matching on $P$ and let
  $(\gamma_0,\gamma_1, \gamma_2,\ldots,\gamma_k)$ be the sequence of
  edges of a directed cycle $\gamma$ of $\Gamma_P^M$.  Then no
  consecutive pair of edges $(\gamma_i, \gamma_{i+1})$---indices
  modulo $k$---has both edges in $M$ or both edges in the complement
  of $M$.
\end{prop}
\begin{proof}
  Following an edge of $M$ causes a unit decrease in the grading.
  Following an edge not in $M$ causes a unit increase in the grading.
  Hence $\gamma$ must contain the same number of edges in $M$ as it
  does edges not in $M$.  Since $M$ is a matching, there is no
  consecutive pair of edges of $\gamma$ both contained in $M$.
  Suppose we have a consecutive pair of edges
  $(\gamma_i,\gamma_{i+1})$ neither of which are contained in $M$.
  Then there are two more $M$-edges in
  $(\gamma_{i+2}, \gamma_{i+3}, \ldots, \gamma_{i+k})$ then there are
  non-$M$-edges.  Hence there is a consecutive pair of edges of
  $\gamma$ both contained in $M$, contradicting the hypothesis that
  $M$ is a matching.
\end{proof}

We require the following basic theorems of Forman discrete Morse
theory.

\begin{thm}[Forman \cite{Forman:1998}]
  \thmlabel{collapse} Let $P$ be a cell poset.  Let $M$ be an acyclic
  matching on $P$ and let $Q$ be the set of critical cells of $M$.  If
  $Q$ is downward closed, then $X_Q$ can be obtained from $X_P$ by a
  sequence of elementary collapses.  In particular, $X_Q$ is homotopy
  equivalent to $X_P$.
\end{thm}

\begin{thm}[{Forman \cite[Corollary~3.5]{Forman:1998}}]
  Let $P$ be a cell poset and let $M$ be an acyclic matching on $P$.
  Then $X_P$ is homotopy equivalent to a CW complex with as many cells
  of each dimension as $M$ has critical cells of that dimension.
\end{thm}

Let $P$ be a finite graded poset.  A \defterm{spherical matching} on
$P$ is an acyclic matching $M$ on $P$ with two critical cells.

\begin{thm}[{Sphere Theorem of Forman \cite[Theorem~5.1(1)]{Forman:1998}}]
  \thmlabel{sphere} Let $P$ be a cell poset.  If $P$ has a spherical
  matching then $X_P$ is homotopy equivalent to a sphere.
\end{thm}

We also require the Generalized Poincar{\'e} Conjecture for
topological manifolds.

\begin{thm}[Topological Generalized Poincar{\'e} Conjecture]
  \thmlabel{pconj} A closed topological manifold $X$ is homotopy
  equivalent to the $d$-sphere iff it is homeomorphic to the
  $d$-sphere.
\end{thm}

The first breakthrough in the proof of the Generalized Poincar{\'e}
Conjecture was made by Smale, who proved that a PL-triangulated
manifold $X$ that is homotopy equivalent to the $d$-sphere is
homeomorphic to the $d$-sphere, for $d \ge 5$ \cite{Smale:1961}.
Stallings gave a different proof of this fact for $d \ge 7$ using an
``engulfing'' method \cite{Stallings:1960}.  This method was later
extended by Zeeman to prove the cases $d = 5$ and $d = 6$
\cite{Zeeman:1962}.  Newman generalized the engulfing method to
topological manifolds and thus completed the proof of the Topological
Generalized Poincar{\'e} Conjecture for $d \ge 5$
\cite[Theorem~7]{Newman:1966}.  In dimension 4 the conjecture was
proved by Freedman \cite{Freedman:1982}.  In dimension 3 it was proved
by Perelman \cite{Perelman:2002a, Perelman:2002b, Perelman:2002c}.  In
dimensions at most 2, it follows from the classification of manifolds.

\begin{thm}
  \thmlabel{plsphere} Let $P$ be a poset such that the order complexes
  of the under sets of $P$ are PL-triangulated spheres.  The order
  complex of $P$ is a PL-triangulated sphere iff the order complexes
  of $P$ and all of its over sets are homotopy equivalent to
  spheres.
\end{thm}

\begin{rmk}
  \rmklabel{lpsphere} Let $P$ be as in \Thmref{plsphere}.  By
  \Thmref{cellpos} and \Thmref{links}, $P$ and its over sets are
  all cell posets.  So if $Q$ is equal to $P$ or to one of its over
  sets then to show that the order complex of $Q$ is homotopy
  equivalent to a sphere it suffices to show that $Q$ has a spherical
  matching.  This holds by \Thmref{sphere} and the fact that the order
  complex of $Q$ is the barycentric subdivision of $X_Q$.
\end{rmk}

\begin{proof}[Proof of \Thmref{plsphere}]
  We prove the ``if'' part since the ``only if'' part is immediate.
  The proof is by induction on the maximum size $k$ of a chain in $P$.
  If $k = 0$ then $P$ is the empty poset and so has the
  PL-triangulated $-1$-sphere (i.e. the empty simplicial complex) as
  its order complex.

  Suppose $k > 0$ and that the theorem holds for all lesser values of
  $k$.  Take $x \in P$.  We show that $O_x$ satisfies the conditions
  of the theorem.  The under sets of $O_x$ are strict intervals of
  $P$ and so, by \Thmref{links}, the order complexes of the under
  sets of $O_x$ are PL-triangulated spheres.  By assumption $O_x$ is
  homotopy equivalent to a sphere as are its over sets since they
  are also over sets of $P$.  Hence, by the inductive hypothesis,
  the order complex of $O_x$ is a PL-triangulated sphere.  By
  \Thmref{links}, the strict intervals of $P$ are also PL-triangulated
  spheres and so, by \Propref{combmfold}, the order complex $O_P$ of
  $P$ is a PL-triangulated manifold.  By assumption, $O_P$ is homotopy
  equivalent to a sphere so, by \Thmref{pconj}, $O_P$ is homeomorphic
  to a sphere and so is a PL-triangulated sphere.
\end{proof}

\section{Bisimplices}
\seclabel{bisimplex}

A \defterm{bipartitioned set} is a set $S$ along with a bipartition
$S = A \sqcup B$ into subsets that are possibly empty.  A subset
$T \subseteq S$ is considered to be a bipartitioned set with its
induced bipartition $T = (T \cap A) \sqcup (T \cap B)$.  Our goal is
to span cells on certain subsets of a bipartitioned set, just as
simplices are spanned on subsets of vertices of a simplicial complex.
In our case, however, not all subsets are eligible to span a cell so
we introduce the term spanworthy.  A bipartitioned set
$S = A \sqcup B$ is \defterm{spanworthy} if $S \neq \emptyset$ and the
following holds. \[ |A| \le 1 \Longleftrightarrow |B| \le 1 \]
Spanworthiness excludes precisely the following cases.
\begin{itemize}
\item $\emptyset \sqcup \emptyset$
\item $A \sqcup \emptyset$ with $|A| \ge 2$
\item $\emptyset \sqcup B$ with $|B| \ge 2$
\item $A \sqcup \{b\}$ with $|A| \ge 2$
\item $\{a\} \sqcup B$ with $|B| \ge 2$
\end{itemize}
In particular, it excludes any $S$ of cardinality $3$.

\begin{thm}
  \thmlabel{bsbd} Let $S = A \sqcup B$ be a spanworthy bipartitioned
  set and let $P$ be the collection of spanworthy proper subsets of
  $S$ ordered by inclusion.  Then $P$ is a cell poset and $O_P$ is a
  PL-triangulated sphere.
\end{thm}
\begin{proof}
  Let $A = \{ a_0, a_1, \ldots, a_m\}$ and
  $B= \{ b_0, b_1, \ldots, b_n\}$.
  
  If $m = -1$ or $n = -1$ then, by spanworthiness, $S$ is a singleton
  and so $P$ is empty.  Then $P$ is the cell poset of the empty
  simplicial complex, i.e., the PL-triangulated $-1$-sphere.  The
  PL-triangulated $-1$-sphere is equal to its own barycentric
  subdivision $O_P$ so the theorem holds in this case.

  Assume that $m \ge 0$ and $n \ge 0$.  If $m = 0$ or $n = 0$ then, by
  spanworthiness, $m = n = 0$ and $P$ is the poset with two
  incomparable elements.  This is the cell poset of the two point
  simplicial complex, i.e., the PL-triangulated $0$-sphere.  The
  PL-triangulated $0$-sphere is equal to its own barycentric
  subdivision $O_P$ so the theorem holds in this case.

  Assume that $m \ge 1$ and $n \ge 1$.  If $m = n = 1$ then the
  elements of $P$ are the singletons and the $\{a_i\}\sqcup\{b_j\}$
  for $i$ and $j$ ranging over $0$ and $1$.  So $P$ is isomorphic to
  the cell poset of the $4$-cycle and so the theorem holds.

  So, by symmetry, we may assume that $m \ge 1$ and $n > 1$.  Assume
  that the theorem holds for all $S$ of lesser cardinality.  Then the
  order complexes of the under sets of $P$ are PL-triangulated
  spheres.  So, by \Thmref{cellpos}, $P$ is a cell poset and, by
  \Thmref{plsphere}, it suffices to show that the order complexes of
  $P$ and all of its over sets are homotopy equivalent to spheres.

  Let $T \in P$.  Spanworthiness implies that $|T| \neq 3$.  If
  $|T| > 3$ then every proper subset of $S$ containing $T$ is
  spanworthy and so $O_T$ is isomorphic to the cell poset of the
  boundary of a simplex and so has order complex homeomorphic to a
  sphere.  So it remains only to show that the order complexes of $P$
  and $O_T$ for $|T| \in \{1,2\}$ are homotopy equivalent to spheres.
  By \Rmkref{lpsphere} it suffices to show that $P$ and such $O_T$
  have spherical matchings.  By symmetry we need only consider the
  cases $T = \{a_0\} \sqcup \emptyset$, $T = \emptyset \sqcup \{b_0\}$
  and $T = \{a_0\} \sqcup \{b_0\}$.

  Consider the following families of edges of the Hasse diagram
  $\Gamma_P$ of $P$.
  \begin{align*} 
    &M_1 = \Biggl\{\hspace{-7pt}&\{a_i\}\sqcup\emptyset &\to \{a_i\}\sqcup\{b_n\} &&& \Biggr\} \\
    &M_2 = \Biggl\{\hspace{-7pt}&\emptyset\sqcup\{b_j\} &\to \{a_m\}\sqcup\{b_j\} &&\sth
    &\text{$j \neq n$}\Biggr\} \\
    &M_3 = \Biggl\{\hspace{-7pt}&\{a_i\}\sqcup\{b_j\} &\to \{a_i, a_m\}\sqcup\{b_j, b_n\} &&\sth
    &\text{$i \neq m$, $j \neq n$} \Biggr\} \\
    &M_4 = \Biggl\{\hspace{-7pt}&A'\sqcup\{b_j,b_n\} &\to
    \bigl(A'\cup \{a_m\}\bigr)\sqcup\{b_j, b_n\} &&\sth
    & \begin{array}{r}
        \text{$a_m \notin A'$, $j \neq n$} \\
        \text{$|A'| \ge 2$}
    \end{array} \Biggr\} \\
    &M_5 = \Biggl\{\hspace{-7pt}&A'\sqcup B' &\to A' \sqcup \bigl(B'\cup\{b_n\}\bigr) &&\sth
    & \begin{array}{r}
        \text{$b_n \notin B'$} \\ 
        \text{$A'\sqcup B' \neq A \sqcup \bigl(B \setminus \{b_n\}\bigr)$} \\
        \text{$|A'| \ge 2$, $|B'| \ge 2$}
      \end{array}
    \Biggr\}
  \end{align*}

  Recall that $B = \{b_0,b_1,\ldots,b_n \}$ and we have assumed
  $n > 1$.  Thus $|B| > 2$ and so the terminal endpoints of edges in
  $M_3$ and $M_4$ are proper subsets of $A \sqcup B$ and hence are
  elements of $P$.

  The endpoints of these edges from different families or from
  different ends of edges in the same family can be distinguished by
  the cardinality of their parts and by the presence of $a_m$ and
  $b_n$, as shown in \Tabref{match}.  Moreover the initial endpoints
  of two edges from the same family are equal if and only if their
  terminal endpoints are equal.  Hence we see that $M = \cup_i M_i$
  forms a matching on $P$.

  \begin{table}
    \centering
    \begin{tabular}{l l l l l l l l l}
      \hline
      Family & \multicolumn{4}{l}{Initial Endpoint}
      & \multicolumn{4}{l}{Terminal Endpoint} \\
      \hline
      & \multicolumn{2}{l}{Cardinalities} & $a_m$ & $b_n$
      & \multicolumn{2}{l}{Cardinalities} & $a_m$ & $b_n$ \\
      $M_1$ & $1$ & $0$ & & $\bot$
      & $1$ & $1$ & & $\top$ \\
      $M_2$ & $0$ & $1$ & $\bot$ & $\bot$
      & $1$ & $1$ & $\top$ & $\bot$ \\
      $M_3$ & $1$ & $1$ & $\bot$ & $\bot$
      & $2$ & $2$ & $\top$ & $\top$ \\
      $M_4$ & $\ge 2$ & $2$ & $\bot$ & $\top$
      & $>2$ & $2$ & $\top$ & $\top$ \\
      $M_5$ & $\ge 2$ & $\ge 2$ & & $\bot$
      & $\ge 2$ & $> 2$ & & $\top$ \\
      \hline
    \end{tabular}
    \caption{Distinguishing characteristics of the endpoints of edges
      in the families of edges described in the proof of
      \Thmref{bsbd}.  Under $a_m$ or $b_n$, the symbol $\top$
      indicates that this element is present in every member of the
      family and the symbol $\bot$ indicates that this element is not
      present in any member of the family.}
    \tablabel{match}
  \end{table}

  Let $\Gamma_P^M$ be the directed graph obtained from $\Gamma_P$ by
  reversing the direction of each edge in $M$.  By
  \Propref{alternate}, to show that $M$ is an acyclic matching we need
  only show that $\Gamma_P^M$ does not contain any directed cycles
  whose edges alternate between being contained and not contained in
  $M$.  To do this it suffices to define a function
  $\alpha\colon P \to \N$ such that $\alpha(T_2) < \alpha(T_0)$ for
  any directed path \[ T_0 \xto{e_0} T_1 \xto{e_1} T_2 \]
  of $\Gamma_P^M$ with $e_0 \in M$ and $e_1 \notin M$.  Note that in
  $\Gamma_P$, $e_0$ is directed from $T_1$ to $T_0$ so we have the
  following inclusions. \[ T_0 \supsetneq T_1 \subsetneq T_2 \]
  We define $\alpha$ as follows. \[ \alpha(T) = %
  \begin{cases}
    0, & \text{$a_m \notin T$ and $b_n \notin T$} \\
    1, & \text{$a_m \in T$ and $b_n \notin T$} \\
    2, & \text{$a_m \notin T$ and $b_n \in T$} \\
    3, & \text{$a_m \in T$ and $b_n \in T$}
  \end{cases} \]
  We may think of $\alpha$ as a function summing the weights on the
  elements of $T$, where $a_m$ is assigned a weight of $1$ and $b_n$
  is assigned a weight of $2$ and all remaining elements have zero
  weight.  Since $T_1 = T_0 \setminus A$ for some nonempty
  $A \subseteq \{a_m,b_n \}$, we have $\alpha(T_1) < \alpha(T_0)$.
  Suppose $\alpha(T_2) = \alpha(T_0)$.  Then
  $T_1 \cup A \subseteq T_2$ and, since
  $|T_0| - |T_1| = |T_2| - |T_1|$, we have $T_0 = T_2$.  This is a
  contradiction since a pair of vertices of $\Gamma_P^M$ may be joined
  by at most one edge and this edge is directed in a unique way.
  Suppose $\alpha(T_2) > \alpha(T_0)$.  Then
  $T_1 = T_0 \setminus \{ a_m \}$ and $T_2 = T_1 \cup \{b_n\}$.  The
  equality $T_1 = T_0 \setminus \{ a_m \}$ implies that $e_0$ is the
  reverse of an edge in $M_2$ or $M_4$.  The
  equality $T_2 = T_1 \cup \{b_n\}$ and the fact that $T_2 \neq T_1$
  implies that $b_n \notin T_1$.  This rules out the possibility that
  $e_0$ is the reverse of an edge in $M_4$.  Thus we have
  \[ \{a_m\}\sqcup\{b_j\} \xto{e_0} \emptyset\sqcup\{b_j\} \]
  and so $T_1 = \emptyset \sqcup \{ b_j \}$ and
  $T_2 = \emptyset \sqcup \{ b_j, b_n \}$ which is not spanworthy, a
  contradiction.  We have established that $M$ is an acyclic matching.
  
  Let $T \in P$.  Since the Hasse diagram of $O_T$ is an induced
  subgraph of $\Gamma_P$, the subset $M_T \subset M$ consisting of all
  edges both of whose endoints are contained in $O_T$ is an acyclic
  matching on $O_T$.  It remains only to show that $M$ is spherical on
  $P$ and that $M_T$ is spherical on $O_T$ for
  $T = \{a_0\} \sqcup \emptyset$, $T = \emptyset \sqcup \{b_0\}$ and
  $T = \{a_0\} \sqcup \{b_0\}$.  In fact it will suffice to prove that
  $M$ is spherical on $P$ with critical elements
  $\emptyset \sqcup \{b_n\}$ and
  $A \sqcup \bigl(B \setminus \{b_n\}\bigr)$.  Indeed, in this case
  the only critical element of $M$ contained in $O_T$ would be
  $A \sqcup \bigl(B \setminus \{b_n\}\bigr)$.  The only other possible
  critical elements of $O_T$ would arise from edges of $M$ having one
  endpoint in $O_T$ and the other endoint in $P \setminus O_T$.  But
  there is a unique such edge of $M$, namely the edge with initial
  endpoint $T$.  Hence $M_T$ would have two critical elements.
  
  We now prove that $M$ is spherical with critical elements
  $\emptyset \sqcup \{b_n\}$ and
  $A \sqcup \bigl(B \setminus \{b_n\}\bigr)$.  First we verify that
  these elements are indeed unmatched by $M$.  Singletons in $B$
  appear as endpoints only in $M_2$ where $\emptyset \sqcup \{b_n\}$
  is not present so $\emptyset \sqcup \{b_n\}$ is critical.  The
  element $A \sqcup \bigl(B \setminus \{b_n\}\bigr)$ is maximal in $P$
  and so may only appear as a terminal endpoint of an edge of $M$.
  These all contain $b_n$ except those in $M_2$ where they have the
  form $\{a_m\} \sqcup \{b_j\}$.  Such an element cannot be equal to
  $A \sqcup \bigl(B \setminus \{b_n\}\bigr)$ since then
  $A \sqcup B = \{a_m\} \sqcup \{b_j,b_n\}$ which is not spanworthy.

  Now, suppose $T = A' \sqcup B'$ is an element of $P$ that is not
  equal to $\emptyset \sqcup \{b_n\}$ or
  $A \sqcup \bigl(B \setminus \{b_n\}\bigr)$.  We will show that $T$
  is matched in $M$.  We consider the following cases separately: (I)
  $|T| = 1$, (II) $|T| = 2$, (III) $|T| = 4$, (IV) $|T|> 4$ and
  $|B'| = 2$, (V) $|T| > 4$ and $|B'| > 2$.

  \paragraph{Case I. $|T| = 1$.}  If $|A'| = 1$ then $T$ is an initial
  endpoint in $M_1$.  Otherwise $|B'| = 1$ and $T$ is an initial
  endpoint in $M_2$.

  \paragraph{Case II. $|T| = 2$.}  Then $|A'| = |B'| = 1$ by
  spanworthiness.  If $b_n \in T$ then $T$ is a terminal endpoint of
  $M_1$.  Otherwise $T$ is a terminal endpoint in $M_2$ if $a_m \in T$
  and $T$ is an initial endpoint in $M_3$ if $a_m \notin T$.

  \paragraph{Case III. $|T| = 4$.}  Then $|A'| = |B'| = 2$ by
  spanworthiness.  If $b_n \notin T$ then $T$ is an initial endpoint
  of $M_5$.  Otherwise $T$ is a terminal endpoint in $M_3$ if
  $a_m \in T$ and $T$ is an initial endpoint in $M_4$ if
  $a_m \notin T$.

  \paragraph{Case IV. $|T|> 4$ and $|B'| = 2$.}  Then $|A'| > 2$.  If
  $b_n \notin T$ then $T$ is an initial endpoint of $M_5$.  Otherwise
  $T$ is a terminal endpoint in $M_4$ if $a_m \in T$ and $T$ is an
  initial endpoint in $M_4$ if $a_m \notin T$.

  \paragraph{Case V. $|T| > 4$ and $|B'| > 2$.}  Then $|A'| \ge 2$ by
  spanworthiness.  If $b_n \in T$ then $T$ is a terminal endpoint in
  $M_5$.  Otherwise $T$ is an initial endpoint in $M_5$.
\end{proof}

\begin{cor}
  \corlabel{bisimplex} Let $S = A \sqcup B$ be a spanworthy
  bipartitioned set and let $P'$ be the collection of spanworthy
  subsets of $S$ ordered by inclusion.  Then $P'$ is the cell poset of
  a regular CW complex homeomorphic to a ball.
\end{cor}
Note that the difference between $P'$ in \Corref{bisimplex} and $P$ in
\Thmref{bsbd} is that $P'$ contains $S$.
\begin{proof}[Proof of \Corref{bisimplex}]
  By \Thmref{cellpos}, it suffices to show that the order complexes of
  the under sets of $P'$ are homeomorphic to spheres.  Since $S$ is
  the maximum in $P'$, the under sets of $P'$ are
  $P = P' \setminus \{S\}$ and the under sets of $P$.  By
  \Thmref{bsbd}, we know that $P$ is the cell poset of a regular CW
  complex $X_P$ that is homeomorphic to a sphere.  Hence, by
  \Thmref{cellpos}, the under sets of $P$ have order complexes
  homeomorphic to spheres and the order complex of $P$ is the
  barycentric subdivision of $X_P$ and so is also homeomorphic to a
  sphere.
\end{proof}

Let $S = A \sqcup B$ be a spanworthy bipartitioned set and let $P'$ be
the collection of spanworthy subsets of $S$ ordered by inclusion.  By
\Corref{bisimplex}, $P'$ is the cell poset of a regular CW complex
$X_{P'}$ that is homeomorphic to a ball.  A regular CW complex
isomorphic to $X_{P'}$ is an \defterm{$(m,n)$-bisimplex} where
$m = |A|-1$ and $n = |B|-1$.  We let $\bs^{m,n}$ denote an
$(m,n)$-bisimplex.  See \Figref{lowdbs}.

\begin{prop}
  There is an isomorphism $\bs^{m,n} \isomor \bs^{n,m}$.
\end{prop}
\begin{proof}
  This is clear from the symmetry of the definition.
\end{proof}

\begin{prop}
  The cells of a bisimplex are all bisimplices.
\end{prop}
\begin{proof}
  The cell poset $P$ of a bisimplex $\bs$ is isomorphic to the poset
  of spanworthy subsets of a spanworthy set $A \sqcup B$.  The cell
  poset of a cell $x$ of $\bs$ corresponds to $P' = U_T \cup \{T\}$
  for some spanworthy $T$.  Hence, the cell poset of $x$ is isomorphic
  to the poset of spanworthy subsets of the spanworthy set $T$.
\end{proof}

\begin{prop}
  \proplabel{bsoneskel} The $1$-skeleton of $\bs^{m,n}$ is isomorphic
  to the complete bipartite graph $K_{m+1,n+1}$ on $m+1$ and $n+1$
  vertices.
\end{prop}
\begin{proof}
  let $P$ be the cell poset of $\bs^{m,n}$ viewed as the poset of
  spanworthy subsets of $A \sqcup B$ with $|A| = m+1$ and $|B| = n+1$.
  The $0$-cells of $\bs^{m,n}$ correspond to the minimal elements of
  $P$.  These are precisely the singletons in $A \sqcup B$ and so we
  may identify the $0$-skeleton of $\bs^{m,n}$ with $A \sqcup B$.  The
  $1$-cells of $\bs^{m,n}$ correspond to those elements of $P$ that
  cover singletons.  These are precisely the sets $\{a\} \sqcup \{b\}$
  with $a \in A$ and $b \in B$.
\end{proof}

\begin{prop}
  Let $m \ge 1$ and let $n \ge 1$.  The dimension of $\bs^{m,n}$ is
  $m+n$.
\end{prop}
\begin{proof}
  Let $P$ be the poset of spanworthy subsets of $A \sqcup B$ with
  $A = \{a_0, a_1, \ldots, a_m\}$ and $B = \{b_0, b_1, \ldots b_n\}$.
  Identify $P$ with the cell poset of $\bs^{m,n}$.  Since $\bs^{m,n}$
  is homeomorphic to a ball of some dimension $k$, the maximal chains
  of $P$ all have cardinality $k+1$.  One such maximal chain is the
  following.
  \begin{align*}
    &\{a_0\} \sqcup \emptyset \subsetneq \{a_0\} \sqcup \{b_0\} \subsetneq
    \{a_0, a_1\} \sqcup \{b_0, b_1\} \\
    &\subsetneq \{a_0, a_1, a_2\} \sqcup \{b_0, b_1\}
    \subsetneq \{a_0, a_1, a_2, a_3\} \sqcup \{b_0, b_1\} \subsetneq \cdots
    \subsetneq \{a_0, \ldots, a_m \} \sqcup \{b_0, b_1\} \\
    &\subsetneq \{a_0, \ldots, a_m\} \sqcup \{b_0, b_1, b_2\} \subsetneq \cdots
    \subsetneq \{a_0, \ldots, a_m \} \sqcup \{b_0, \ldots, b_n\}
  \end{align*}
  This chain has cardinality $|A| + |B| - 1$ where the $-1$ is due to
  the jump
  $\{a_0\} \sqcup \{b_0\} \subsetneq \{a_0, a_1\} \sqcup \{b_0,
  b_1\}$.  Hence $\bs^{m,n}$ has dimension
  $k+1-1 = |A|+|B|-1 - 1 = m+ n$.
\end{proof}

\section{Bisimplicial Complexes}
\seclabel{bscomplex}

A full subcomplex $Y$ of a regular CW complex $X$ is \defterm{full} if
$\bd x \subset Y$ implies $x \subset Y$ for any cell $x$ of $X$.  A
\defterm{bisimplicial complex} is a regular CW complex $X$ such that
each cell $x$ of $X$ is isomorphic to a bisimplex and, for any two
bisimplices $x$ and $y$ of $X$, the intersection $x \cap y$ is a full
subcomplex of $X$.  Note that this implies that the bisimplices
themselves are full subcomplexes and, furthermore, that any finite
intersection of bisimplices is full.

A complete bipartite graph $K$ is \defterm{spanworthy} if it is
nonempty, connected and the bipartition on its vertex set is
spanworthy.  A spanworthy complete bipartite subgraph $K$ of the
$1$-skeleton $X^1$ of a bisimplicial complex $X$ \defterm{spans} a
bisimplex $\bs$ of $X$ if the $1$-skeleton $(\bs)^1$ of $\bs$ is equal
to $K$.  Note that at most one bisimplex may span $K$ since the
intersection of two distinct bisimplices $\bs$ and $\bs'$ spanning $K$
would be full in neither $\bs$ nor $\bs'$.  A bisimplicial complex $X$
is \defterm{flag} if every spanworthy complete bipartite subgraph $K$
of $X^1$ spans a bisimplex $\bs$.  We use the notation $\bs(A;B)$ to
denote $\bs$, where $A \sqcup B$ is the bipartitioned vertex set of
$K$.

\begin{defn}
  \defnlabel{bscomp} Let $\Gamma$ be a graph.  The \defterm{flag
    bisimplicial completion} $\bs(\Gamma)$ of $\Gamma$ is a flag
  bisimplicial complex defined inductively as follows.  The
  $1$-skeleton of $\bs(\Gamma)$ is $\Gamma$.  Now, assume the
  $(k-1)$-skeleton of $\bs(\Gamma)$ has been defined.  The
  $k$-skeleton is obtained by the following operation.  To each
  subcomplex isomorphic to some $\bd \bs^{m,n}$ with
  $\dim(\bs^{m,n}) = k$, glue in a copy of $\bs^{m,n}$ along the
  isomorphism.
\end{defn}

Note that if $X$ is a flag bisimplicial complex then $X = \bs(X^1)$.

Let $\Gamma$ be a finite bipartite graph.  We view $\Gamma^0$ as a
metric space with the shortest path metric.  The \defterm{metric
  sphere} $S_r(u) \subseteq \Gamma^0$ of radius $r$ about
$u \in \Gamma^0$ is the set of vertices of $\Gamma$ at distance $r$
from $u$.  If $u$ and $v$ are distinct vertices of $\Gamma$ then $u$
is \defterm{dominated} by $v$ if there is an inclusion
$S_1(u) \subset S_1(v)$ of neighbourhoods.

A finite bipartite graph $\Gamma$ is \defterm{bi-dismantlable} if
there exists a sequence
$\Gamma = \Gamma_1, \Gamma_2, \ldots, \Gamma_n$ of graphs ending on a
nonempty connected complete bipartite graph such that, for each
$i < n$, $\Gamma_{i+1}$ is a subgraph of $\Gamma_i$ induced on the
complement of $\{v_i\}$ for some $v_i$ dominated in $\Gamma_i$.

\begin{thm}
  \thmlabel{dismclps} Let $X$ be a finite flag bisimplicial complex
  with $X^1$ bipartite.  If $X^1$ is bi-dismantlable then $X$ is
  collapsible.
\end{thm}
\begin{proof}
  The proof is by induction on the length of the bi-dismantling
  sequence.

  In the base case, $X^1$ is a nonempty connected complete bipartite
  graph on some bipartitioned vertex set $S = A \sqcup B$.  Let
  $A = \{ a_0, \ldots, a_m \}$ and $B = \{ b_0, \ldots, b_n\}$.
  Without loss of generality $|A| \le |B|$.  If $X^1$ is not
  spanworthy then, as it is nonempty and connected, we have $|A| = 1$
  and $|B| \ge 2$.  Then the only spanworthy subgraphs of $X^1$ are
  its edges and vertices and so $X = X^1$.  But $X^1$ is a tree and so
  $X$ is collapsible.

  Suppose now that $X^1$ is spanworthy.  By flagness, $X$ is a
  bisimplex $\bs(A;B)$.  If $|A| + |B| \le 4$ then $\bs(A;B)$ is
  isomorphic to a vertex, an edge or a square.  These are collapsible.
  So we assume that $|A| \ge 2$ and $|B| > 2$.  Identify the cell
  poset of $\bs(A;B)$ with the poset $P'$ of nonempty spanworthy
  subsets of $A \sqcup B$.  Then the poset
  $P = P' \setminus \{ A \sqcup B \}$ is the cell poset of
  $\bd \bs(A;B)$.  The proof of \Thmref{bsbd} gives a spherical
  matching $M$ on $P$.  Let $M'$ be the matching obtained from $M$ by
  adding the following edge.
  \[ A \sqcup \bigl(B \setminus \{b_n\}\bigr) \to A \sqcup B \] Then
  $M'$ is acyclic and leaves only $\emptyset \sqcup \{ b_n \}$
  unmatched.  Hence, by \Thmref{collapse}, $X = \bs(A;B)$ is
  collapsible.
  
  Now, suppose suppose the bi-dismantling sequence has nonzero length
  with $v$ the first dominated vertex in the sequence.  Let $u$ be a
  dominator of $v$ in $X^1$.  Let $P$ be the cell poset of
  $X$. Consider the downward closed subset $Q$ of $P$ defined as
  follows.  \[ Q = \{ x \in P \sth x \not\ge v\} \]
  Then the subcomplex $X_Q = \bigcup Q$ is the full subcomplex of $X$
  induced on $X^0 \setminus \{v\}$ and so $X_Q$ is flag.  Moreover,
  $X_Q^1$ is the induced subgraph of $X^1$ obtained from $X^1$ by
  deleting $v$ and so $X_Q^1$ is dismantlable.  Hence $X_Q^1$ is
  collapsible by induction.  Therefore, by \Thmref{collapse}, it
  suffices to construct an acyclic matching $M$ on $P$ whose set of
  critical elements is $Q$.
  
  Let $w$ be any neighbour of $v$ in $X^1$.  Note that the vertices of
  any connected complete bipartite subgraph of $X^1$ containing $v$
  are at distance at most $2$ from $v$.  Consider the following
  families of edges in the Hasse diagram $\Gamma_P$ of $P$.
  \begin{align*}
    &M_1 = \Biggl\{&\bs\bigl(\{v\};\emptyset\bigr)
    &\to \bs\bigl(\{v\};\{w\}\bigr) &&& \Biggr\} \\
    &M_2 = \Biggl\{&\bs\bigl(\{v\};\{x\}\bigr)
    &\to \bs\bigl(\{u,v\};\{w,x\}\bigr) &&\sth
    &\text{$x \in S_1(v) \setminus \{w\}$} \Biggr\} \\
    &M_3 = \Biggl\{&\bs\bigl(\{u,v\};N\bigr)
    &\to \bs\bigl(\{u,v\};\{w\}\cup N\bigr) &&\sth
    & \begin{array}{r}
        \text{$N \subseteq S_1(v) \setminus \{w\}$} \\
        \text{$|N| \ge 2$}
    \end{array} \Biggr\} \\
    &M_4 = \Biggl\{&\bs\bigl(T\cup \{v\};N\bigr)
    &\to \bs\bigl(T \cup \{u,v\}; N\bigr) &&\sth
    & \begin{array}{r}
        \text{$T \subseteq S_2(v) \setminus \{u\}$} \\
        \text{$N \subseteq \bigcap_{y\in T\cup \{v\}} S_1(y)$} \\
        \text{$|T| \ge 1$, $|N| \ge 2$}
    \end{array} \Biggr\}
  \end{align*}
  The union $M = \bigcup_i M_i$ is an acyclic matching on $P$ whose
  set of critical elements is $Q$.  The argument is very similar to
  that in the proof of \Thmref{bsbd} with $w$ playing the role of
  $a_m$ and $u$ playing the role of $b_n$.
\end{proof}

\section{Quadric Complexes and Asphericity}
\seclabel{quadric}

\begin{figure}
  \centering
  \includegraphics[width=0.15\textwidth]{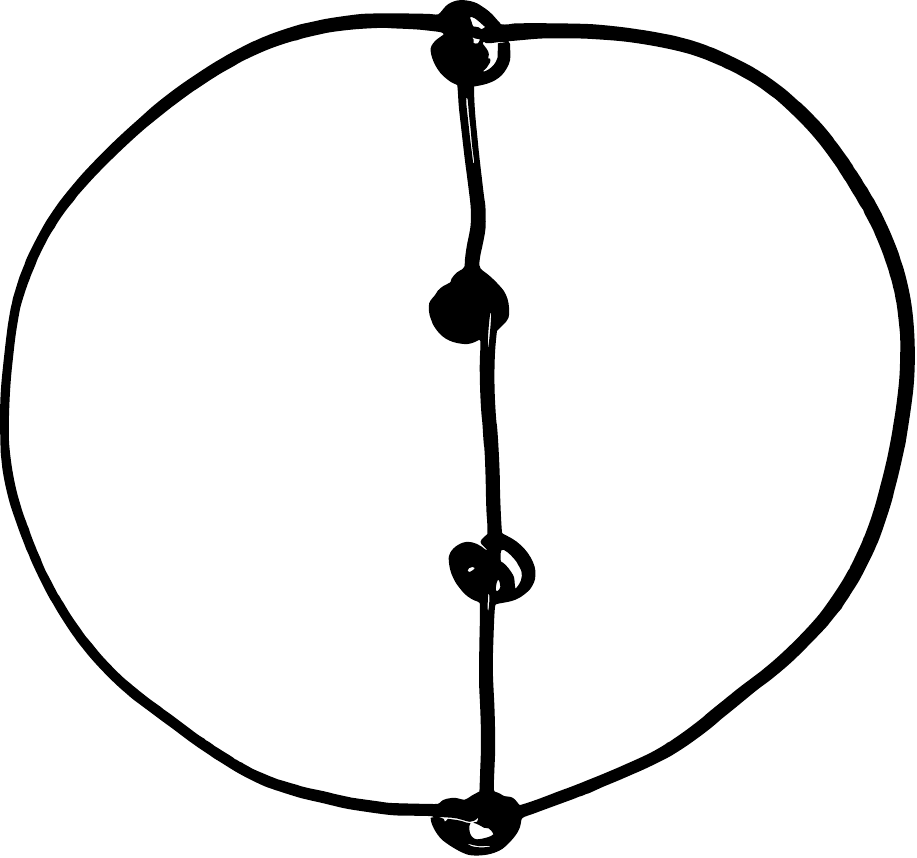}
  \caption{A disk diagram which can not be reduced in a locally
    quadric complex.}
  \figlabel{piec3}
\end{figure}

\begin{figure}
  \centering
  \includegraphics[width=0.4\textwidth]{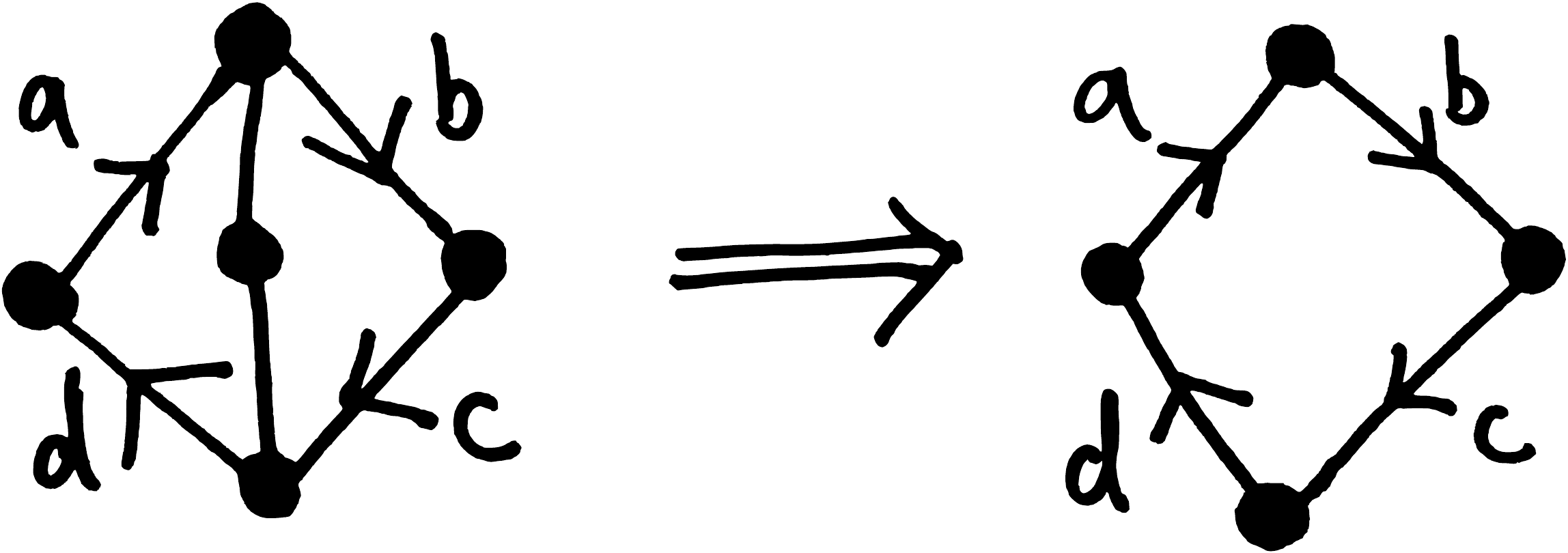}
  \\
  \includegraphics[width=0.8\textwidth]{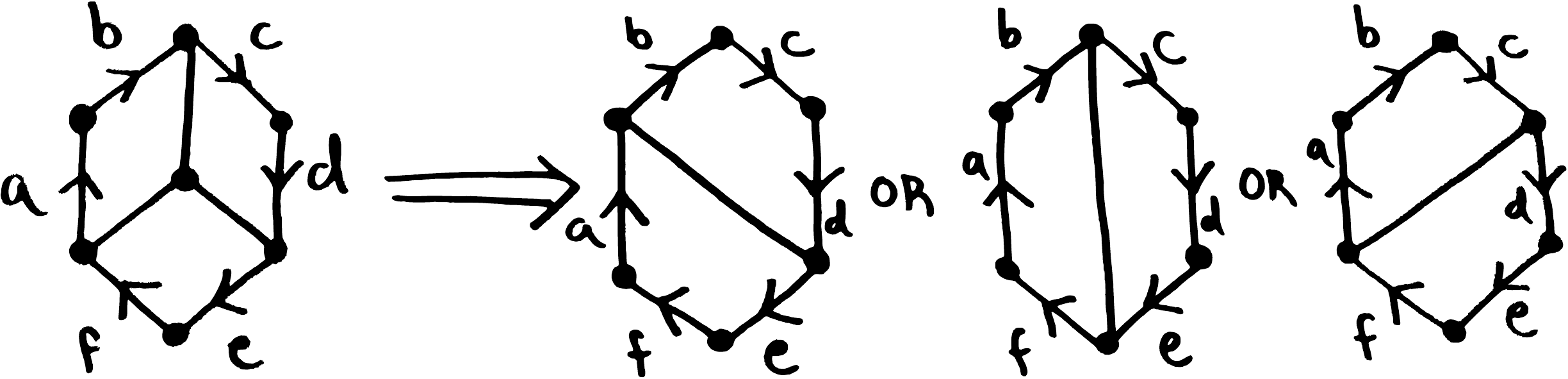}
  \caption{Replacement rules for disk diagrams in quadric complexes.}
  \figlabel{repl}
\end{figure}

\begin{defn}
  \defnlabel{quadric} A \defterm{locally quadric complex} is a locally
  finite square complex $X$ with immersed cells such that no reduced
  disk diagram in $X$ has the form of \Figref{piec3} and any immersed
  disk diagram of a form on the left-hand side of \Figref{repl} has a
  replacement on the right-hand side with the same boundary path.  If,
  in addition, $X$ is simply connected then $X$ is \defterm{quadric}.
  A group $G$ is \defterm{quadric} if it acts properly and cocompactly
  on a quadric complex.
\end{defn}

For a full introduction to quadric complexes and groups see prior work
of the present author \cite{Hoda:2017}.

A square complex $X$ is \defterm{flag} if each square of $X$ is
bounded by an embedded $4$-cycle and each embedded $4$-cycle of $X^1$
bounds a unique square of $X$.

\begin{prop}[{\cite[Proposition~1.18]{Hoda:2017}}]
  \proplabel{quadcycles} Let $X$ be a connected square complex.  Then
  $X$ is quadric if and only if $X$ is flag and every isometrically
  embedded cycle of $X^1$ has length $4$.
\end{prop}

It follows from \Propref{quadcycles} that a quadric complex $X$ is the
$2$-skeleton of the flag bisimplicial completion $\bs(X^1)$.  See
\Defnref{bscomp}.

\begin{thm}[{Bandelt \cite[Theorem~1]{Bandelt:1988}}]
  \thmlabel{bbhm} A graph is hereditary modular if and only if it is
  connected and every isometrically embedded cycle has length $4$.
\end{thm}

A graph is \defterm{modular} if for every triple of vertices $u,v,w$
there exists a vertex $x$ which lies on some geodesic between each
pair of vertices in the triple.  A graph is \defterm{hereditary
  modular} if each of its isometrically embedded subgraphs is modular.

The \defterm{metric ball} of radius $r \in \N$ centered at a vertex
$v$ of a graph (bisimplicial complex) is the induced (full) subgraph
(subcomplex) on the set of vertices of distance at most $r$ to $v$ (in
the $1$-skeleton).

\begin{rmk}
  \rmklabel{hmballs} Let $\Gamma$ be a modular graph.  Then the metric
  balls of $\Gamma$ are isometrically embedded.  In particular, if
  $\Gamma$ is hereditary modular then its metric balls are hereditary
  modular.
\end{rmk}

\begin{thm}[{Bandelt \cite[Theorem~2]{Bandelt:1988}}]
  \thmlabel{hmdism} Let $\Gamma$ be a finite nonempty hereditary
  modular graph.  Then $\Gamma$ is bi-dismantlable.
\end{thm}

\begin{thm}
  \thmlabel{quadcont} Let $X$ be a nonempty quadric complex.  Then the
  flag bisimplicial completion $\bs(X^1)$ is contractible.
\end{thm}
\begin{proof}
  The metric balls of $X$ are quadric by \Propref{quadcycles},
  \Thmref{bbhm} and \Rmkref{hmballs}.  These balls are finite since
  quadric complexes are locally finite.  Hence balls in $X$ are
  collapsible by \Propref{quadcycles}, \Thmref{bbhm}, \Thmref{hmdism}
  and \Thmref{dismclps}.  The balls of $X$ centered at a fixed vertex
  give an ascending exhaustion of $X$ by contractible subcomplex and
  so $X$ is contractible.
\end{proof}

\subsection{A \texorpdfstring{$K(G,1)$}{K(G,1)} for Torsion-Free
  Quadric Groups}

Let $X$ be a locally quadric complex.  Then the universal cover
$\univcov X$ is quadric and so $\pi_1 X$ is quadric.  Let
$\square_{m,n}$ denote the $2$-skeleton of the bisimplex $\bs^{m,n}$.
Let $\square_{m,n} \to X$ be an immersion with $m,n \ge 2$.  Since
$\square_{m,n}$ is simply connected it lifts to $\univcov X$.  Since
quadric complexes do not contain loops or bigons \cite{Hoda:2017},
every lift $\square_{m,n} \to \univcov X$ is an embedding.  Every
torsion-free quadric group is the fundamental group of a compact
locally quadric complex.  However, a compact locally quadric complex
may have a fundamental group with torsion.  The following theorem
allows to understand when this is the case.

\begin{thm}[Invariant Biclique Theorem \cite{Hoda:2017}]
  Let $F$ be a finite group acting on a quadric complex $\univcov X$.
  Then $F$ stabilizes a nonempty connected complete bipartite subgraph
  of $\univcov X$.
\end{thm}

To state the following corollary we need a definition.  Let
$\Aut(\square_{m,n})$ be the set of automorphisms of $\square_{m,n}$.
Note that $\Aut(\square_{m,n})$ acts on the set of immersions
$\{ \square_{m,n} \to X \}$.  For a given immersion
$\square_{m,n} \to X$, we define \defterm{$\Aut(\square_{m,n} \to X)$}
as the stabilizer of $\square_{m,n} \to X$ in $\Aut(\square_{m,n})$.

\begin{cor}
  \corlabel{immtorsion} Let $X$ be a compact locally quadric complex.
  Then $\pi_1(X)$ has torsion if and only if
  $\Aut(\square_{m,n} \to X)$ is nontrivial for some immersion
  $\square_{m,n} \to X$ with $m,n \ge 1$.
\end{cor}
\begin{proof}
  Suppose $g \in \pi_1(X) \setminus \{1 \}$ has finite order.  Then
  $\subgp{g}$ stabilizes a nonempty connected complete bipartite
  subgraph $K_{m+1,n+1}$ of $\univcov X^1$.  Since the action is free,
  we have $m,n \ge 1$.  The full subcomplex induced by this
  $K_{m+1,n+1}$ is a $\square_{m,n}$.  Restricting the covering map to
  $\square_{m,n}$ we have an immersion $\square_{m,n} \to X$.
  Restricting the action of $g$ to $\square_{m,n}$ we obtain a
  nontrivial automorphism of $\square_{m,n} \to X$.

  Now, suppose there is a nontrivial automorphism
  $\phi \colon \square_{m,n} \to \square_{m,n}$ of an immersion
  $\square_{m,n} \to X$.  Let $f \colon \square_{m,n} \to \univcov X$
  be a lift of this immersion and identify $\square_{m,n}$ with its
  image under $f$.  Then $\phi$ extends to a nontrivial deck
  transformation which must have finite order.
\end{proof}

Let $X$ be a compact locally quadric complex.  We will construct a
compact complex $X^{+}$ having $X$ as its $2$-skeleton such that
$X^{+}$ is a $K\bigl(\pi_1(X),1\bigr)$ if $\pi_1(X)$ is torsion-free.
The construction is by induction on dimension.  The $2$-skeleton of
$X^{+}$ is $X$.  Now suppose we have already constructed the
$(k-1)$-skeleton $(X^{+})^{k-1}$ of $X^{+}$.  To obtain the
$k$-skeleton of $X^{+}$ perform the following operation.  Along each
immersion $\bd\bs^{m,n} \to (X^{+})^{k-1}$ with $\dim(\bs^{m,n}) = k$
glue in a copy of $\bs^{m,n}$.  For the purposes of this operation,
two immersions which are isomorphic over $(X^{+})^{k-1}$ are
considered identical and so result in only a single gluing.  Since $X$
is compact there is a bound on the size of a connected complete
bipartite graph which can immerse in $X$.  Hence $X^{+}$ is compact.

\begin{lem}
  \lemlabel{xplbsc} Let $X$ be a compact locally quadric complex with
  $\pi_1(X)$ torsion-free.  Then the universal cover $\univcov{X^{+}}$
  is isomorphic to the bisimplicial completion
  $\bs\bigl({\univcov X}^1\bigr)$.
\end{lem}
\begin{proof}
  The proof is by induction on skeleta.  The $2$-skeleta of $X^{+}$
  and $\bs({\univcov X}^1)$ are $X$ and $\univcov X$ so the base case
  holds.  Assume the statement holds for the $(k-1)$-skeleta:
  $\bigl(\univcov{X^{+}}\bigr)^{k-1} \isomor \bs\bigl(\univcov
  X^1\bigr)^{k-1}$.  Each $\bd \bs^{m,n}$ subcomplex,
  $\dim(\bs^{m,n})=k$, of $\bs\bigl({\univcov X}^1\bigr)^{k-1}$
  immerses into $(X^{+})^{k-1}$ under the covering map and so spans a
  $\bs^{m,n}$ in $\bigl(\univcov{X^{+}}\bigr)^k$.  On the other hand,
  each immersion $\bd \bs^{m,n} \to (X^{+})^{k-1}$ with
  $\dim(\bs^{m,n})=k$ lifts to an embedding in
  $\bs\bigl({\univcov X}^1\bigr)^{k-1}$ whose image thus spans a
  unique $\bs^{m,n}$ in $\bs\bigl({\univcov X}^1\bigr)^k$.  So the set
  of boundaries of the $k$-dimensional bisimplices of
  $\bigl(\univcov{X^{+}}\bigr)^k$ and
  $\bs\bigl({\univcov X}^1\bigr)^k$.  No two $k$-dimensional
  bisimplices have the same boundary in
  $\bs\bigl({\univcov X}^1\bigr)^k$.  The same holds for
  $\bigl(\univcov{X^{+}}\bigr)^k$ by \Corref{immtorsion} and so
  $\bigl(\univcov{X^{+}}\bigr)^k \isomor \bs\bigl(\univcov
  X^1\bigr)^k$.
\end{proof}

The main theorem of this section follows immediately from
\Lemref{xplbsc} and \Thmref{quadcont}.

\begin{thm}
  \thmlabel{quadkgone} Let $X$ be a compact locally quadric complex.
  If $\pi_1(X)$ is torsion-free then $X^{+}$ is a compact
  $K(\pi_1(X),1)$.
\end{thm}

\bibliographystyle{abbrv}
\bibliography{nima}
\end{document}